\documentclass[a4paper,11pt]{amsart}
\usepackage{amscd,amsfonts,amssymb}
\usepackage{tikz}
\usepackage{pgfplots}
\usetikzlibrary{automata,positioning,calc,trees}
\usetikzlibrary{intersections,pgfplots.fillbetween}
\usepackage{graphicx,tikz}
\usepackage{pgf,tikz,pgfplots}
\usepackage{mathrsfs}
\usetikzlibrary{arrows}
\usepackage{enumerate}
\usepackage[shortlabels]{enumitem}
\usepackage{mathrsfs}
\usepackage{amssymb,amsmath,amsthm,color}
\usepackage{caption,subcaption}
\usepackage{hyperref}
\usepackage{url}
\usepackage{setspace}
\usepackage{float}
\newtheorem{thm}{Theorem}[section]
\newtheorem{coro}[thm]{Corollary}
\newtheorem{lemma}[thm]{Lemma}
\newtheorem{prop}[thm]{Proposition}
\theoremstyle{definition}

\theoremstyle{remark}

\numberwithin{equation}{section}
\def\be#1 {\begin{equation} \label{#1}}
	\newcommand{\ee}{\end{equation}}

\def\sqw{\hbox{\rlap{\leavevmode\raise.3ex\hbox{$\sqcap$}}$%
		\sqcup$}}
\def\findem{\ifmmode\sqw\else{\ifhmode\unskip\fi\nobreak\hfil
		\penalty50\hskip1em\null\nobreak\hfil\sqw
		\parfillskip=0pt\finalhyphendemerits=0\endgraf}\fi}

\newcommand{\R}{{\mathbb {R}}}

\newcommand{\N}{{\mathbb N}}
\newcommand{\Z}{{\mathbb Z}}

\newcommand{\supp}{\operatorname{supp}}
\begin{document}
	\baselineskip16pt
	
	\title[]{Sparse bounds for maximal oscillatory rough singular integral operators}
	\author[ ]{ Surjeet Singh Choudhary, Saurabh Shrivastava and Kalachand Shuin}
	
	\address{(Surjeet Singh Choudhary and Saurabh Shrivastava) Department of Mathematics, Indian Institute of Science Education and Research Bhopal, Bhopal-462066, India}
	\email{ surjeet19@iiserb.ac.in, saurabhk@iiserb.ac.in}
	
	\address{(Kalachand Shuin)Department of Mathematical Sciences, Seoul National University, Seoul 08826, Republic of Korea}
	\email{ kcshuin21@snu.ac.kr}
	
	\date{\today}

	\begin{abstract}
		We prove sparse bounds for maximal oscillatory rough singular integral operator 
		$$T^{P}_{\Omega,*}f(x):=\sup_{\epsilon>0} \left|\int_{|x-y|>\epsilon}e^{\iota P(x,y)}\frac{\Omega\big((x-y)/|x-y|\big)}{|x-y|^{n}}f(y)dy\right|,$$
		where $P(x,y)$ is a real-valued polynomial on $\mathbb{R}^{n}\times \mathbb{R}^{n}$ and $\Omega\in L^{\infty}(\mathbb{S}^{n-1})$ is a homogeneous function of degree zero with $\int_{\mathbb{S}^{n-1}}\Omega(\theta)~d\theta=0$. This allows us to conclude weighted $L^p-$estimates for the operator $T^{P}_{\Omega,*}$. Moreover, the norm $\|T^P_{\Omega,*}\|_{L^p\rightarrow L^p}$ depends only on the total degree of the polynomial $P(x,y)$, but not on the coefficients of $P(x,y)$. Finally, we will show that these techniques also apply to obtain sparse bounds for oscillatory rough singular integral operator $T^{P}_{\Omega}$ for $\Omega\in L^{q}(\mathbb{S}^{n-1})$, $1<q\leq\infty$.
	\end{abstract}
	
	\maketitle
	
	\section{Introduction } \label{section1}
The study of oscillatory singular integrals has its origin in PDEs, the study of singular Radon transforms and Hilbert transform along curves. Let us briefly describe the case of Radon transform and Hilbert transform along curves to motivate the study of oscillatory singular integrals.  For, consider a point $P=(x,t)\in \R^{n+1}$ with $x\in \R^n$ and $t\in \R$ and let $L_{(x,t)}:y\rightarrow(x+y,t+\langle Bx,y\rangle)$, $y\in\mathbb{R}^{n}$ be a linear map, where $B=(b_{jk})$ is a real-valued bilinear form. Note that the image of $L_{(x,t)}$ defines a hypersurface in $\mathbb{R}^{n+1}$ which passes through $P$. The Radon transform associated to a suitable kernel $K$ on $\mathbb{R}^{n}$ is defined by 
	$$\mathcal{R}u(x,t)=\int_{\mathbb{R}^{n}}u(L_{(x,t)}(y))K(y)~dy.=\int_{\mathbb{R}^{n}}u(x+y,t+\langle Bx,y\rangle)K(y)~dy.$$
	Using the Fourier transform, we can express the operator $\mathcal{R}$ as 
	\begin{eqnarray*}
		\mathcal{R}u(x,t)&=&\frac{1}{2\pi}\int_{\mathbb{R}} e^{\iota \lambda t} \int_{\mathbb{R}^{n}}e^{\iota \lambda \langle Bx,y\rangle}K(y)\hat{u}(x+y,\lambda)~dy d\lambda\\
		&=&\int_{\R} e^{\iota \lambda t}e^{-\iota \lambda \langle Bx,x\rangle}(T_{\lambda}\hat{u}(\cdot,\lambda))(x)~d\lambda,
	\end{eqnarray*}
	where $$T_{\lambda}f(x)=\int_{\R^n} e^{\iota \lambda \langle Bx,y\rangle}K(y-x)f(y)~dy, $$
	for $f\in C^{\infty}_{0}(\mathbb{R}^{n})$. Invoking the Plancherel theorem, $L^{2}$-estimate for the Radon transform $\mathcal{R}$ can be deduced by proving $L^{2}$-estimates for the operator $T_{\lambda}$. Observe that the operator $T_{\lambda}$ has an oscillatory factor with the kernel. We refer to~Phong and Stein~\cite{PhongStein} for more details. 
	
	We recall another interesting aspect of studying oscillatory singular integrals through Hilbert transform along curves. For example, recall that the Hilbert transform along a suitable plane curve $\Gamma=\{(t,\gamma(t)):t\in \R\}$ is defined by 
	$$H_{\gamma}f(x)=p.v.\int_{\R} f(x_1-t,x_2-\gamma(t))\frac{dt}{t}.$$
	In particular, if we consider $\gamma(t)=t^{\alpha}, 0<\alpha\neq1,$ and take $f(x_1,x_2)=g(x_1)e^{\iota x_2}\chi_{\{|x_2|\leq R\}}$, then the case of $R\rightarrow\infty$ corresponds to the study of oscillatory singular integral associated with the kernel $\frac{e^{\iota |x|^{\alpha}}}{x}$. We refer to ~Chanillo, Kurtz and Sampson~\cite{CKS1,CKS2} for more details on $H_{\gamma}$ when $\gamma(t)=t^{\alpha}, 0<\alpha\neq1$.
	
	Next, we briefly recall some of the important developments concerning $L^p-$estimates for the oscillatory singular integral operators.  Ricci and Stein~\cite{Ricci} obtained the first results for oscillatory singular integral operators  
	\begin{eqnarray}
		T^P_{K}f(x):=\int_{\mathbb{R}^{n}}e^{\iota P(x,y)}K(x,y)f(y)~dy,~x\notin supp(f) ,  
	\end{eqnarray}
	where $K\in C^1(\R^n\setminus\{0\})$ is a Calder\'{o}n-Zygmund kernel.  They proved that $T^P_{K}$ extends to a bounded operator on $L^p(\R^n)$ for all $1<p<\infty$. Moreover, the norm $\|T^P_{K}\|_{L^p\rightarrow L^p}$ depends only on the total degree of the polynomial $P(x,y)$, but not on the coefficients of $P(x,y)$. Later, Chanillo and Christ ~ \cite{CC} proved that the operator $T^P_{K}$ is of weak-type $(1,1)$.  Al-Qassem, Cheng and Pan~\cite{Qassem1, Qassem2} extended the $L^p$  estimates to oscillatory integrals with kernels satisfying Hölder type regularity condition. In~\cite{Qassem1} they proved the results with the oscillatory factor $e^{\iota B(x,y)}$, where $B(x,y)$ is a real-valued bilinear form. Later, in \cite{Qassem2} they managed to prove bounds with the oscillatory factor $e^{\iota P(x,y)}$, where $P(x,y)$ is a real-valued polynomial in $\mathbb{R}^{n}\times \mathbb{R}^{n}$. Later, Lacey and Spencer \cite{Lacey} established sparse bounds for the operator $T^{P}_{K}$ for Calder\'{o}n-Zygmund kernels $K\in C^1(\R^n\setminus\{0\})$. The sparse bounds for the maximally truncated oscillatory singular integrals $T^{P}_{K,*}$ for Calder\'{o}n-Zygmund kernels $K\in C^1(\R^n\setminus\{0\})$ were obtained in~\cite{KL}. We would like to remark here that the sparse bounds in~\cite{KL} imply the weak-type $(1,1)$ boundedness of $T^{P}_{K,*}$. 
	
	In this paper, we are concerned with oscillatory singular integral operators with rough kernels. Let $\Omega$  be a homogeneous function of degree zero on the unit sphere $\mathbb{S}^{n-1}, n\geq 2,$ with mean-value zero, i.e.,  $\int_{\mathbb{S}^{n-1}}\Omega(\theta)~d\theta=0$. Denote  $K(y)=\frac{\Omega(y')}{|y|^{n}}, y'=y/|y|\in \mathbb{S}^{n-1}$. Let $P(x,y)$ be a real-valued polynomial on $\mathbb{R}^{n}\times \mathbb{R}^{n}$. The oscillatory singular integral operator associated with the kernel $K$ and polynomial $P$ is defined by  
	\begin{eqnarray}
		T^P_{\Omega}f(x):=p.v\int_{\mathbb{R}^{n}}e^{\iota P(x,y)}K(x-y)f(y)~dy,
	\end{eqnarray}
	where $f$ is a compactly supported smooth function defined on $\R^n$. The maximally truncated oscillatory singular integral operator is defined by 
	$$T^{P}_{\Omega,*}f(x):=\sup_{\epsilon>0}\Big| \int_{|x-y|>\epsilon}e^{\iota P(x,y)}K(x-y)f(y)dy\Big|.$$
	As mentioned we are interested in operators $T^P_{\Omega}$ with rough kernels, i.e., we assume that $\Omega$ belongs to some $ L^q(\mathbb{S}^{n-1}), 1\leq q\leq \infty$ and does not satisfy any regularity condition. The kernel $K$ associated with such an $\Omega$ is called a rough kernel and the corresponding operator is called a rough singular integral operator.  
	
	Observe that if the polynomial $P(x,y)$ is of the form $P_1(x)+P_2(y)$, where $P_1$ and $P_2$ are polynomials in $x$ and $y$ respectively, the $L^p$ estimates for $T^P_{\Omega}f$ and $T^{P}_{\Omega,*}f$ follow immediately from those of the corresponding rough singular integral operators
	\begin{eqnarray}
		T_{\Omega}f(x):=p.v\int_{\mathbb{R}^{n}}K(x-y)f(y)~dy,
	\end{eqnarray} and 
	$$T_{\Omega,*}f(x):=\sup_{\epsilon>0}\Big| \int_{|x-y|>\epsilon}K(x-y)f(y)dy\Big|.$$
	In view of this remark, henceforth we shall assume that the polynomial $P(x,y)$ does not take the form $P_1(x)+P_2(y)$. There is a vast literature available in the context of rough singular integral operators $T_{\Omega}$ (which is $T^P_{\Omega}$ without oscillatory factor). We refer the reader to Calder\'{o}n and Zygmund~ \cite{Calderon}, Christ and Rubio de Francia~\cite{Christ}, Hofmann in \cite{Hofmann}, Seeger~\cite{Seeger}, Conde-Alonso et al. ~\cite{Conde} and  Hyt\'{o}nen et al. in \cite{Hytonen} for a systematic account, including recent developments, of rough singular integral operators. 
	
	We would like to highlight here that unlike the case of $T^P_{K}$ with $C^1(\R^n\setminus\{0\})$ continuous kernels, the $L^p$ estimates for $T^P_{\Omega}$ are not well understood for rough kernels. Shanzhen and Yan~ \cite{LZ} obtained the following $L^p-$estimates in this direction. 
	\begin{thm}\cite{LZ}\label{LZ}
		Let $\Omega\in L^q(\mathbb{S}^{n-1}), 1<q\leq\infty$. Then $T^{P}_{\Omega}$ is bounded on $L^{p}(\R^n)$ for all $ 1<p<\infty$, with the operator norm of $T^{P}_{\Omega}$ independent of the coefficients of $P(x, y)$.
	\end{thm}
	In the same paper Shanzhen and Yan~ \cite{LZ} also proved that $T^{P}_{\Omega,*}$ is bounded in $L^p(\R^n), 1<p<\infty$ with $\Omega\in L^{\infty}(\mathbb{S}^{n-1}).$ 
	Jiang and Lu~\cite{JL} improved the condition that $\Omega \in L^q(\mathbb{S}^{n-1}), 1<q\leq \infty$ in Theorem~\ref{LZ} to $\Omega\in L\log L(S^{n-1})$. Recently,  Chen and Tao~\cite{CT} obtained weighted $L^{p}-$estimates for the operator $T^{P}_{\Omega,*}$ for $1<p<\infty$ where $\Omega\in L^\infty(\mathbb{S}^{n-1})$.  Motivated by these developments about oscillatory rough singular integrals we study sparse bounds for oscillatory rough singular integral operators $T^{P}_{\Omega}$ and $T^{P}_{\Omega,*}$. 
	
	\subsection{Main results}\label{results}
	The aim of this paper is to establish sparse bounds for the oscillatory rough singular integral operators $T^{P}_{\Omega}$ and $T^{P}_{\Omega,*}$. It is well-known that sparse bounds strengthen the weighted $L^p-$estimates for the operator with quantitative bounds on the operator norm and yield vector-valued results as a consequence. In this section, we will state the main results establishing sparse bounds for the operators under consideration. This requires some preparation. 
	
	\subsection*{Sparse family of cubes:} Let $ \mathcal D _{t}, 1\leq t \leq 3 ^{n}$ denote the collection of shifted dyadic cubes. We know that the collection 
	\begin{equation*}
		\{ \tfrac 13  Q  \;:\; Q\in \mathcal D_t,\  l(Q)= 2 ^{j} \}
	\end{equation*}
	forms a partition of $ \mathbb R ^{n}$ for each $j$. We will simply refer to cubes in~$ \mathcal D _{t}$ as dyadic cubes and will work with a fixed collection $ \mathcal D _{t}$. The reader is invited to see~\cite{Lerner1} for details about dyadic cubes. 
	
	Let $\mathcal{S}$ be a family of dyadic cubes in $\mathbb{R}^{n}$ and  $\eta \in (0,1)$. We say that $\mathcal{S}$ is an $\eta-$sparse family of cubes if for every $Q\in\mathcal{S}$ there exists a subset $E_{Q}\subseteq Q$ such that $|E_{Q}|>\eta|Q|$ and the collection $\{E_{Q}\}$ is pairwise disjoint. We will simply refer to $\eta-$sparse family as sparse family. 
	
	Given a sparse family $\mathcal{S}$ and exponents $1\leq r,s<\infty$, the $(r,s)-$sparse bilinear form is defined by 
	$$\Lambda_{\mathcal{S},r,s}(f,g):=\sum_{Q\in\mathcal{S}}|Q|\bigg(\frac{1}{|3Q|}\int_{3Q}|f(x)|^{r}dx\bigg)^{1/r}\bigg(\frac{1}{|3Q|}\int_{3Q}|g(x)|^{s}dx\bigg)^{1/s},$$
	where $f$ and $g$ are compactly supported bounded functions.  
	
	\subsection*{$A_p$ weights} Let $\omega$ be a non-negative locally integrable function. Then $\omega\in A_{p}$ for $1<p<\infty$, if there exists a constant $C>0$ such that 
	\begin{eqnarray*}
		[\omega]_{A_{p}}:=\sup_{Q}\bigg(\frac{1}{|Q|}\int_{Q}\omega(x)~dx\bigg)\bigg(\frac{1}{|Q|}\int_{Q}\sigma(x)~dx\bigg)^{p-1}<C,
	\end{eqnarray*}
	where $p'=\frac{p}{p-1}$ and $\sigma(x)=\omega(x)^{1-p'}$.  The $A_{\infty}$ weight characteristic was introduced by Fujii \cite{Fujii} and  Wilson \cite{Wilson}. It is defined as follows 
	\begin{eqnarray*}
		[\omega]_{A_{\infty}}:=\sup_{Q}\big(\int_{Q}\omega(x)~dx\bigg)^{-1}\int_{Q} M(\chi_{Q}\omega)(x)~dx, 
	\end{eqnarray*}
	where $M$ is the Hardy-Littlewood maximal function.
	
	We have the following sparse bounds for the maximal operator $T^{P}_{\Omega,*}$.
	\begin{thm}\label{mainresult}
		Let $\Omega\in L^{\infty}(\mathbb{S}^{n-1})$ be a homogeneous function of degree zero with $\int_{\mathbb{S}^{n-1}}\Omega(\theta)~d\theta=0$.  Let $P(x,y)$ be a real-valued polynomial on $\mathbb{R}^{n}\times \mathbb{R}^{n}$ with total degree $d$. Then for any exponent $r$ with $1<r<2$ and compactly supported bounded functions $f$ and $g$, the following sparse domination holds
		\begin{eqnarray}\label{sparsedomination}
			\lvert  \langle T^{P}_{\Omega,*} f ,g  \rangle\rvert  \leq C \sup_{\mathcal{S}}\Lambda _{\mathcal{S},r,r} (f,g).
		\end{eqnarray}
		Here the constant $C=C(\|T_{\Omega,*}\|_{L^p\rightarrow L^p}, d, n, r)$ is independent of the coefficients of $P(x,y)$. 
	\end{thm}
	Next result proves sparse domination for the operator $T^{P}_{\Omega}$.
	\begin{thm}\label{Lq}
		Let $\Omega\in L^{p}(\mathbb{S}^{n-1}), p>1$, be a homogeneous function of degree zero with $\int_{\mathbb{S}^{n-1}}\Omega(\theta)~d\theta=0$. Let $P(x,y)$ be a real-valued polynomial on $\mathbb{R}^{n}\times \mathbb{R}^{n}$ with total degree $d$. Then for any $r,s$ with $(2p)'<r<2,~ p'\leq s$ and compactly supported bounded functions $f$ and $g$,  the following sparse domination result holds
		\begin{eqnarray}\label{sparsedomination1}
			\lvert  \langle T^{P}_{\Omega} f ,g  \rangle\rvert  \leq C \sup_{\mathcal{S}}\Lambda _{\mathcal{S},r,s} (f,g).
		\end{eqnarray}
		Here the constant $C=C(\|T_{\Omega}\|_{L^p\rightarrow L^p}, d, n, r)$  is independent of the coefficients of $P(x,y)$.     
	\end{thm}		
	As mentioned before, sparse domination implies weighted and vector-valued estimates for the underlying operator. We record here the weighted result for the maximal operator $T^{P}_{\Omega,*}$, which can be obtained as a consequence of Theorem~\ref{mainresult}. We refer the reader to~ \cite{Nieraeth} for obtaining vector-valued results as a consequence of sparse domination. 
	\begin{coro}\label{weightedresult}
		Let $1<p<\infty$ and $\omega\in A_{p}$. Let $\Omega\in L^{\infty}(\mathbb{S}^{n-1})$ be a homogeneous function with $\int_{\mathbb{S}^{n-1}}\Omega(\theta)~d\theta=0$  and $P(x,y)$ be a real-valued polynomial on $\mathbb{R}^{n}\times \mathbb{R}^{n}$.  Then we have 
		$$\Vert T^{P}_{\Omega,*}f\Vert_{L^{p}(\omega)}\leq C[\omega]^{1/p}_{A_{p}}\big([\omega]^{1/p'}_{A_{\infty}}+[\sigma]^{1/p}_{A_{\infty}}\big)\max\{[\sigma]_{A_{\infty}},[\omega]_{A_{\infty}}\}\Vert f\Vert_{L^{p}(\omega)}.$$
	\end{coro}
	The quantitative bounds on weighted characteristic in the theorem above follow from  \cite{Hytonen}. Note that this recovers the weighted estimates proved in~\cite{CT}.
	
	\section{Auxiliary results}
	The proof of sparse domination result in Theorem~\ref{mainresult} uses the ideas from Lacey and Spencer~ \cite{Lacey}. However, note that in Lacey and Spencer~ \cite{Lacey} it is important that the kernel $K$ satisfies $C^1$ regularity condition away from the origin. The regularity property in the case of a rough kernel is not available. We get around this difficulty by making use of the ideas from Shanzhen and Yan~ \cite{LZ}. 
	
	First, we recall a well-known fact about sparse bilinear forms. We will require this result in our proofs. 
	\begin{prop}[\cite{Mena}, Section $4$]\label{universalsparse}
		Given $1\leq r,s<\infty$, and  compactly supported bounded functions $f,g$, there exists a single sparse form $\Lambda_{\mathcal{S}_{0},r,s}$ such that 
		\begin{eqnarray*}
			\sup_{\mathcal{S}}\Lambda_{\mathcal{S},r,s}(f,g)\lesssim \Lambda_{\mathcal{S}_{0},r,s}(f,g).
		\end{eqnarray*}  
	\end{prop}
	In this section, we develop results that are required in proving Theorems~\ref{mainresult} and \ref{Lq}. In order to prove the desired $L^p$ bounds, we decompose the original operator into dyadic blocks away from the origin. More precisely, we write 
	\begin{equation*}
		\chi_0+\sum_{j=2} ^{\infty} \chi_j  =1,
	\end{equation*}
	where $\chi_j=\chi_{[2^{j-1},2^{j})}$ for $j\geq2$ and $\chi_0=\chi_{[0,2)}$.  This decomposes the operator  $T_{\Omega}^{P}=T_{0}^{P}+\sum\limits_{j\geq 2} T_{j}^{P}$. The same decomposition is carried out to deal with the maximal operator. We will prove suitable estimates for each piece $T_{j,*}^{P}$.  
	
	Let us first consider the local version in the context of truncated maximal rough singular integral operator defined by 
	\begin{eqnarray*}
		T_{\Omega,*}^{0}f(x):=\sup_{0<\epsilon<2}\Big|\int_{\epsilon<|x-y|<2}\frac{\Omega((x-y)')}{|x-y|^{n}}f(y)~dy\Big|,
	\end{eqnarray*}
	where $f$ is a compactly supported smooth function. The following sparse domination result holds for the operator $T_{\Omega,*}^{0}$.
	\begin{lemma}\label{localsparse}
		Given $1<r<2$ and  compactly supported bounded functions $f$ and $g$, there exists a positive constant $C$ such that the following  sparse domination holds 
		\begin{eqnarray}
			|\langle T_{\Omega,*}^{0}f,g\rangle|\leq C\sup_{\mathcal{S}}\sum_{Q\in\mathcal{S}}|Q|\langle f\rangle_{Q,r}\langle g\rangle_{Q,r}.
		\end{eqnarray}
	\end{lemma}
	\begin{proof} 
		The proof of this result uses the idea from~\cite{LZ}. We decompose the function $f$ suitably into three parts. This allows us to compare the singular part of $T_{\Omega,*}^{0}$ with the corresponding rough singular integral operator (without oscillatory factor). We invoke the sparse domination result for the rough singular integral operator $T_{\Omega,*}$ from~\cite{Hytonen} to get the desired bounds for the singular part. The other parts are dealt with using the kernel estimates directly. More precisely, we consider the following.  
		
		For a fixed $h\in\mathbb{R}^{n}$ decompose $f=f_{1}+f_{2}+f_{3},$ where $f_{1}(y)=f(y)\chi_{\{|y-h|<1\}},~f_{2}(y)=f(y)\chi_{\{1<|y-h|<5/2\}}$ and $f_{3}(y)=f(y)\chi_{\{|y-h|\geq 5/2\}}$. Denote $g_{h}(x)=g(x)\chi_{\{|x-h|<1/2\}}$. 
		
		First, observe that  for $|x-h|<1/2$ we  have  $T_{\Omega,*}^0f_{3}(x)=0$. Therefore, the desired estimate holds trivially for this part. 
		
		Next, consider $T_{\Omega,*}^0f_{2}$. Note that for $|x-h|<1/2$ and $1\leq |y-h|<5/2$ we get that  $1/2<|y|<3$. Therefore,  
		\begin{eqnarray*}
			|T_{\Omega,*}^{0}f_{2}(x)|&\leq& \int_{1/2<|y|<3}\frac{|\Omega(y')|}{|y|^{n}}|f_{2}|(x-y)~dy\\
			&\leq& 2^{n}\Vert \Omega\Vert_{{\infty}}\bigg(\int_{|y|<3}|f_{2}(x-y)|~dy\bigg)\\
			&\leq& C_{n}\Vert \Omega\Vert_{{\infty}} Mf_{2}(x),
		\end{eqnarray*}
		where $M$ denotes the Hardy-Littlewood maximal function. Using sparse domination of the Hardy-Littlewood maximal function, we get that 
		\begin{eqnarray*}
			|\langle T_{\Omega,*}^{0}f_{2},g_{h}\rangle|
			&\leq&C_{n}\Vert \Omega\Vert_{{\infty}} \sum_{Q\in\mathcal{S}_{0}}|Q|\langle f_{2}\rangle_{Q,1}\langle g_h\rangle_{Q,1}.
		\end{eqnarray*}
		Since the estimate above holds uniformly in $h$, we get the result for the function $g$ by writing down $g=\sum\limits_{h\in \Z^n} g_h$, i.e. we have that 
		\begin{eqnarray}
			|\langle T_{\Omega,*}^{0}f_{2},g\rangle|&\leq &C_{n}\Vert \Omega\Vert_{{\infty}} \sum_{Q\in\mathcal{S}_{0}}|Q|\langle f\rangle_{Q,1}\langle g\rangle_{Q,1}.
		\end{eqnarray}
		
		Finally, we consider the term $T_{\Omega,*}^{0}f_{1}$. Observe that for $|x-h|<1/2$ it is easy to verify that 
		$$T_{\Omega,*}^{0}f_{1}(x)=T_{\Omega,*}f_{1}(x).$$ Invoking the sparse domination for $T_{\Omega,*}$ from \cite{Hytonen}, for every fixed $1<r<2$ we get that 
		\begin{eqnarray*}
			|\langle T_{\Omega,*}^{0}f_{1},g_{h}\rangle |&=&|\langle T_{\Omega,*}f_{1},g_{h}\rangle |\\
			&\leq&C\sup_{\mathcal{S}}\sum_{Q\in\mathcal{S}}|Q|\langle f_{1}\rangle_{Q,r}\langle g_{h}\rangle_{Q,r}.
		\end{eqnarray*}
		Here the constant $C$ is independent of $h$. 
		
		Recall that we need to prove the estimate above for an arbitrary compactly supported bounded function $g$. Unlike the situation in the case of  $T_{\Omega,*}^{0}f_{2}$ as above, here we have got $L^r$ averages of $g_h$ with $r>1$. It requires a little work to pass to $g$ from $g_h$ in this case. 
		
		Applying Proposition \ref{universalsparse} we get that there exists a sparse family $\mathcal S_0$ such that 
		\begin{eqnarray*}
			|\langle T_{\Omega,*}^{0}f_{1},g_{h}\rangle |&\leq& C\sum_{Q\in\mathcal{S}_{0}}|Q|\langle f_{1}\rangle_{Q,r}\langle g_{h}\rangle_{Q,r}\\
			&=& \sum_{k\geq1}\sum_{Q\in\mathcal{S}_{0},|Q|=2^{kn}}|Q|\langle f_{1}\rangle_{Q,r}\langle g_{h}\rangle_{Q,r}+\sum_{Q\in\mathcal{S}_{0},|Q|\leq1}|Q|\langle f_{1}\rangle_{Q,r}\langle g_{h}\rangle_{Q,r}
		\end{eqnarray*}
		Since $\supp(f_1)\subset B(h,1)$ and $\supp(g_h)\subset B(h,1/2)$, we note that the cubes of length $2^k, k\geq1$ can intersect non-trivially with $\supp(f_1)$ and $\supp(g_h)$ atmost $2^n$ times. For any such cube $Q$, we can find a cube $Q_0\subset Q$ with $l_{Q_{0}}\leq\sqrt{2}$ such that
		$$\supp(f_1)\cap Q\subset\supp(f_1)\cap Q_0~\text{~~~and~~~}\supp(g_h)\cap Q\subset\supp(g_h)\cap Q_0.$$
		Hence, we get
		\begin{eqnarray*}
			|Q|\langle f_{1}\rangle_{Q,r}\langle g_{h}\rangle_{Q,r}&\leq&C_{n}2^{kn}|Q_0|\left(\frac{1}{2^{kn}|Q_0|}\int_{Q_0}|f_1|^r\right)^{\frac{1}{r}}\left(\frac{1}{2^{kn}|Q_0|}\int_{Q_0}|g_h|^r\right)^{\frac{1}{r}}\\
			&\leq& C_{n}2^{kn(1-\frac{2}{r})}|Q_0|\langle f_{1}\rangle_{Q_0,r}\langle g_{h}\rangle_{Q_0,r}
		\end{eqnarray*}
		This gives the following estimate
		\begin{eqnarray*}
			|\langle T_{\Omega,*}^{0}f_{1},g_{h}\rangle |&\lesssim& \sum_{k\geq1}\sum_{Q\in\mathcal{S},l_{Q}=\sqrt{2}}2^{kn(1-\frac{2}{r})}|Q|\langle f_{1}\rangle_{Q,r}\langle g_{h}\rangle_{Q,r}+\sum_{Q\in\mathcal{S},l_{Q}\leq1}|Q|\langle f_{1}\rangle_{Q,r}\langle g_{h}\rangle_{Q,r}\\
			&\lesssim& \sum_{Q\in\mathcal{S},l_{Q}\leq\sqrt{2}}|Q|\langle f_{1}\rangle_{Q,r}\langle g_{h}\rangle_{Q,r}.
		\end{eqnarray*}
		Here we have used that $1<r<2$. 
		We can finish the proof by writing $g=\sum\limits_{h\in \Z^n} g_h$  as in the previous case. 
	\end{proof}
	Next, we consider the operator for each $j\geq 2$. We need some notation. 
	Let $Q$ be a  dyadic cube in $\mathbb{R}^{n}$ with side length $l_{Q}=2^{j+2}$ for $j\in\mathbb{Z}$. Consider  
	\begin{eqnarray*}
		I_{Q}f(x)&:=&\int_{\R^n} e^{\iota P(x,x-y)} \frac{\Omega(y)\chi_{ j} (|y|)} {|y|^{n}}\left(\chi_{\tfrac 13Q} f\right)(x-y)\; dy.
	\end{eqnarray*}
	
	Observe that $supp(I_{Q}f)\subset Q$. Using polar decomposition, we can write  
	\begin{eqnarray*}
		I_{Q}f(x)&=&\int_{\mathbb{S}^{n-1}}\Omega(y')\int_{2^{j-1}\leq r<2^{j}}e^{\iota P(x,x-ry')}\frac{f\chi_{\frac{1}{3}Q}(x-ry')}{r}~dr d\sigma(y')
	\end{eqnarray*}
	Note that for any $x\in\mathbb{R}^{n}$ and  fixed $y'\in\mathbb{S}^{n-1}$, there exists a hyperplane  $Y$ which is normal to vector $y'$ and passes through the origin such that $x=z+sy'$ for $z\in Y$ and $s\in\mathbb{R}$. Therefore, we can write 
	\begin{eqnarray}
		I_{Q}f(x) \nonumber &=&\int_{\mathbb{S}^{n-1}}\Omega(y')\int_{2^{j-1}\leq r<2^{j}}e^{\iota P(z+sy',z+sy'-ry')}\frac{f\chi_{\frac{1}{3}Q}(z+sy'-ry')}{r}~dr d\sigma(y')\\
		\nonumber &=&\int_{\mathbb{S}^{n-1}}\Omega(y')\int_{2^{j-1}\leq s-t<2^{j}}e^{\iota P(z+sy',z+ty')}\frac{f\chi_{\frac{1}{3}Q}(z+ty')}{s-t}~dr d\sigma(y')\\
		\label{NQ}	&=&\int_{\mathbb{S}^{n-1}}\Omega(y')N_{Q}[f(z+\cdot y')](s).
	\end{eqnarray}
	We have the following $L^2$ bounds for the operator $N_Q$ defined as above. 
	\begin{lemma}\label{mainlemma}
		Let $P(x,y)=\sum\limits_{|\alpha|\leq m,
			|\beta|\leq l}a_{\alpha,\beta}x^{\alpha}y^{\beta}$ be a real-valued polynomial defined on $\R^n\times \R^n$. Then for any dyadic cube $Q$ with side length $l_{Q}=2^{j+2}$, there exists a $\delta>0$ such that  
		\begin{eqnarray*}
			\Vert N_{Q}\Vert_{L^{2}(\mathbb{R})\rightarrow L^{2}(\mathbb{R})}\lesssim 2^{-j\delta/2}\bigg|\sum_{|\alpha|=m,|\beta|=l}a_{\alpha,\beta}y'^{\alpha+\beta}\bigg|^{-\delta/2m},
		\end{eqnarray*} 
		where $y'=y/|y|\in\mathbb{S}^{n-1}$.
	\end{lemma}
	\begin{proof}
		The proof is based on a $T^*T$ argument. We adapt the ideas from~\cite{LZ} to complete our proof. Let $N^{*}_{Q}$ denote the adjoint of the linear operator $N_{Q}$. Note that the kernel of composition operator $N^{*}_{Q}N_{Q}$ is given by 
		\begin{eqnarray*}
			\mathcal{K}_{Q}(u,v):=\int_{\substack{2^{j-1}\leq r-v<2^{j}\\2^{j-1}\leq r-u<2^{j}}}e^{\iota\big(P(z+ry',z+vy')-P(z+ry',z+uy')\big)} \frac{1}{(r-u)(r-v)}~dr.
		\end{eqnarray*}
		Using a change the variables argument, we can write
		\begin{eqnarray*}
			\mathcal{K}_{Q}(u,v)&=&\int_{\substack{1/2\leq r<1\\2^{j-1}\leq 2^jr+v-u<2^{j}}}e^{\iota\big(P(z+2^{j}ry'+vy',z+vy')-P(z+2^{j}ry'+vy',z+uy')\big)} \psi(r)~dr,
		\end{eqnarray*}
		where for fixed $u$ and $v$ we denote 
		$$\psi(r)=\psi_{u,v}(r)=\frac{1}{r(2^{j} r+v-u)}.$$
		Now, it is easy to verify that the following estimate holds. 
		\begin{equation}\label{trivial}
			|\mathcal{K}_{Q}(u,v)|\leq\frac{C}{2^{j}} \chi_{\left[0,2^{j-1}\right]}(|v-u|).
		\end{equation}
		Next, we organize the oscillatory factor in $\mathcal{K}_{Q}$ by writing the polynomial $P(x,y)$ as follows
		$$P(x,y)=\sum_{|\alpha|=m} x^{\alpha}Q_{\alpha}(y)+R(x,y),$$
		where $R(x, y)$ is a polynomial with degree less than $m$ in the $x$ variable and $Q_{\alpha}(y)$ is a polynomial with degree $l$ in $y$ variable. So we can write
		$$\mathcal{K}_{Q}(u,v)=\int_{\substack{\frac{1}{2}\leq r<1\\2^{j-1}<2^{j}r+v-u \leq 2^{j}}} e^{\iota(E+F)} \psi(r) dr,$$
		where
		$$E=(2^{j} r)^{m} \sum_{|\alpha|=m} y'^{\alpha}[Q_{\alpha}(z+vy')-Q_{\alpha}\left(z+uy'\right)],$$
		and $F$ is a polynomial of degree less than $m$ in the variable $r$. 
		
		The Van der Corput lemma from \cite{Stein} gives us the following estimate. 
		$$\left|\int_{1/2}^{1} e^{\iota(E+F)} dr\right|\leq C\left(2^{jm}\left|\sum_{|\alpha|=m}y'^{\alpha}[Q_{\alpha}(z+vy')-Q_{\alpha}(z+uy')]\right|\right)^{-\frac{1}{m}}$$
		Next, integration by parts argument yields the following estimate on the kernel. 
		\begin{eqnarray*}
			|\mathcal{K}_{Q}(u,v)|&\leq&C\left(2^{jm}\left|\sum_{|\alpha|=m}y'^{\alpha}[Q_{\alpha}(z+vy')-Q_{\alpha}(z+uy')]\right|\right)^{-\frac{1}{m}}\left(\|\psi\|_{\infty}+\int_{\substack{1/2\leq r<1\\2^{j-1}\leq 2^jr+v-u<2^{j}}}|\psi'|dr\right)\\
			&\leq& C2^{-j}\left(2^{jm}\left|\sum_{|\alpha|=m}y'^{\alpha}[Q_{\alpha}(z+vy')-Q_{\alpha}(z+uy')]\right|\right)^{-\frac{1}{m}}
		\end{eqnarray*}
		The inequality above along with the estimate \eqref{trivial} implies that  uniformly in $\delta \in(0,1]$,
		we have 
		$$|\mathcal{K}_{Q}(u,v)|\leq C2^{-j}\left(2^{jm}\left|\sum_{|\alpha|=m}y'^{\alpha}[Q_{\alpha}(z+vy')-Q_{\alpha}(z+uy')]\right|\right)^{-\frac{\delta}{m}}\chi_{\left[0,2^{j-1}\right]}(|v-u|).$$
		Therefore, 
		\begin{eqnarray*}
			\int|\mathcal{K}_{Q}(u,v)|dv&=&\int_{|v-u|<2^{j-1}}|\mathcal{K}_{Q}(u,v)|dv\\
			&\leq& C 2^{-j} 2^{-j \delta} \int_{|v-u|<2^{j}}\left|\sum_{|\alpha|=m}y'^{\alpha}[Q_{\alpha}(z+vy')-Q_{\alpha}(z+uy')]\right|^{-\frac{\delta}{m}}dv\\
			&=& C 2^{-j \delta}\int_{|v|<1}\left|\sum_{|\alpha|=m} y'^{\alpha}[Q_{\alpha}\left(z+2^{j}\left(v+\frac{u}{2^{j}}\right)y'\right)-Q_{\alpha}(z+uy')]\right|^{-\frac{\delta}{m}} dv.
		\end{eqnarray*}
		As in~\cite{LZ}, we use the following result from~\cite{Ricci} to estimate the integral in the inequality above. 
		\begin{lemma}\cite{Ricci}\label{poly}
			Suppose $p(x)=\sum\limits_{|\alpha| \leq d} a_{\alpha} x^{\alpha}$ is a polynomial of degree $l$ and $\varepsilon<1 / l$. Then
			$$\sup _{y \in \mathbb{R}^{n}} \int_{|x| \leq 1}|p(x-y)|^{-\varepsilon} d x \leq A_{\varepsilon}\left(\sum_{|\alpha|=l}\left|a_{\alpha}\right|\right)^{-\varepsilon}.$$
		\end{lemma}
		Applying this result with the choice of $\delta \in(0,1]$ such that $\frac{\delta}{m}<\frac{1}{l}$ we obtain
		\begin{eqnarray*}
			\int|\mathcal{K}_{Q}(u,v)|dv&\leq& C2^{-j\delta}\left|\sum_{|\alpha|=m,|\beta|=l} a_{\alpha,\beta}y'^{\alpha+\beta}\right|^{-\frac{\delta}{m}} 2^{-jl \frac{\delta}{m}} \\
			&\leq& C 2^{-j\delta}\left|\sum_{|\alpha|=m,|\beta|=l} a_{\alpha,\beta}y'^{\alpha+\beta}\right|^{-\frac{\delta}{m}}
		\end{eqnarray*}
		Observe that 
		$$\|N^{*}_{Q}N_{Q}\|_{L^{2}(\mathbb{R})\rightarrow L^{2}(\mathbb{R})}\leq \|\mathcal{K}_{Q}\|_1\lesssim   2^{-j\delta}\left|\sum_{|\alpha|=m,|\beta|=l} a_{\alpha,\beta}y'^{\alpha+\beta}\right|^{-\frac{\delta}{m}},$$
		where  $0<\delta<\frac{m}{l}$. This completes the proof of the lemma.
	\end{proof}
	
	\section{Proof of Theorem \ref{mainresult}}
	In view of Proposition~\ref{universalsparse}, it is enough to prove that for $1<r<2$ we have 
	\begin{eqnarray}\label{sparsedomination1}
		\lvert  \langle T^{P}_{\Omega,*} f ,g  \rangle\rvert  \lesssim \sup_{\mathcal{S}} \Lambda _{\mathcal{S},r,r} (f,g).
	\end{eqnarray}
	The proof uses an induction argument on the degree of the polynomial in $x$ and $y$ variables. 	
	Let $m$ and $l$ be natural numbers. Assume that desired sparse domination for $T^{P}_{\Omega,*}$ holds if the polynomial $P$ is a sum of monomials with degree in $x$ less than $m$ and degree in $y$ any finite positive integer. Similarly, we assume that sparse domination for $T^{P}_{\Omega,*}$ holds if  $P$ is a sum of monomials with degree in $y$ less than $l$ and degree in $x$ any positive integer, i.e. we assume that for polynomials in both the cases as above, we have 
	$$|\langle T^{P}_{\Omega,*}f,g\rangle|\lesssim \sup_{\mathcal{S}}\Lambda_{\mathcal{S},r,r}(f,g),~~\text{for any}~~1<r<2.$$
	We consider the polynomial $P(x,y)$ and rewrite it as follows  $P(x,y)=\sum\limits_{|\alpha|=m,|\beta|=l}a_{\alpha,\beta}x^{\alpha}y^{\beta}+R_{0}(x,y)$, where the polynomial $R_{0}$ satisfies the induction hypothesis. We shall prove the sparse domination for $T^{P}_{\Omega,*}$ associated with $P$. 
	
	First, observe that without loss of generality, we may assume that $\sum\limits_{|\alpha|=m,|\beta|=l}|a_{\alpha,\beta}|>0$. Further, it is easy to observe using a standard dilation argument that it is enough to prove the result with the assumption that $\sum\limits_{|\alpha|=m,|\beta|=l}|a_{\alpha,\beta}|=1$. For, let $A=\left(\sum\limits_{|\alpha|=m,|\beta|=l}\left|a_{\alpha,\beta}\right|\right)^{\frac{1}{l+m}}$ and rewrite the polynomial as 
	$$P(x,y)=\sum_{|\alpha|=m,|\beta|=l}\frac{a_{\alpha,\beta}}{A^{l+m}}(Ax)^{\alpha}(Ay)^{\beta}+R_{0}\left(\frac{A x}{A}, \frac{A y}{A}\right) = Q(Ax,Ay)$$
	Also, note that 
	\begin{eqnarray*}
		T^{P}_{\Omega,\epsilon}f(x)
		&=&T^{Q}_{\Omega,\frac{\epsilon}{A}}f\left(\frac{\cdot}{A}\right)(Ax).
	\end{eqnarray*}
	We use sparse domination with the dilated function and observe that the sparse domination is independent of dilation to conclude the assertion. 
	
	Let us assume that $\sum\limits_{|\alpha|=m,|\beta|=l}|a_{\alpha,\beta}|=1$ and consider 
	\begin{eqnarray*}
		T^P_{\Omega,*}f(x)&\leq& \sup_{0<\epsilon<2}\left|\int_{|x-y|>\epsilon}e^{\iota P(x,y)}\frac{\Omega((x-y)')}{|x-y|^{n}}f(y)~dy\right|\\
		&&+\sup_{\epsilon\geq2}\left|\int_{|x-y|>\epsilon}e^{\iota P(x,y)}\frac{\Omega((x-y)')}{|x-y|^{n}}f(y)~dy\right|\\
		&\leq& \sup_{0<\epsilon<2}\left|\int_{\epsilon<|x-y|<2}e^{\iota P(x,y)}\frac{\Omega((x-y)')}{|x-y|^{n}}f(y)~dy\right|\\
		&&+\left|\int_{|x-y|\geq2}e^{\iota P(x,y)}\frac{\Omega((x-y)')}{|x-y|^{n}}f(y)~dy\right|\\
		&&+\sup_{\epsilon\geq2}\left|\int_{|x-y|>\epsilon}e^{\iota P(x,y)}\frac{\Omega((x-y)')}{|x-y|^{n}}f(y)~dy\right|\\
		&=& T^P_{0,*}f(x)+T^P_{\infty}f(x)+T^P_{\infty,*}f(x)
	\end{eqnarray*}
	We shall prove the desired result for each of the terms as above separately. Let $h\in\mathbb{R}^{n}$. Note that we can rewrite the polynomial as 
	$$P(x,y)=\sum_{|\alpha|=m,|\beta|=l}a_{\alpha,\beta}(x-h)^{\alpha}(y-h)^{\beta}+R(x,y,h),$$
	where the polynomial $R$ satisfies the induction hypothesis.
	We first deal with the local part. We decompose it further in the following way.  \begin{eqnarray}
		T^{P}_{0,*}(f)(x):=T^{P}_{01,*}(f)(x)+T^{P}_{02,*}(f)(x),
	\end{eqnarray}
	where 
	\begin{eqnarray*}
		&&T^{P}_{01,*}(f)(x)=\sup_{0<\epsilon<2}\left|\int_{\epsilon<|x-y|<2}e^{\iota[R(x,y,h)+\sum_{|\alpha|=m,|\beta|=l}a_{\alpha,\beta}(y-h)^{\alpha+\beta}]}\frac{\Omega((x-y)')}{|x-y|^{n}}f(y)~dy\right|\\
		\text{and}~~	&&T^{P}_{02,*}f(x):=\sup_{0<\epsilon<2}\left|\int_{\epsilon<|x-y|<2}\bigg(e^{\iota P(x,y)}-e^{\iota[R(x,y,h)+\sum_{|\alpha|=m,|\beta|=l}a_{\alpha,\beta}(y-h)^{\alpha+\beta}]}\bigg)\frac{\Omega((x-y)')}{|x-y|^{n}}f(y)~dy\right|.
	\end{eqnarray*}
	Note that Lemma~ \ref{localsparse} ~yields the desired sparse bounds for the first term, i.e. we get 
	$$|\langle T^{P}_{01,*}f,g\rangle |\lesssim \sup_{\mathcal{S}}\Lambda_{\mathcal{S},r,r}(f,g).$$
	For the second term $T^{P}_{02,*}f$ observe that if $|x-h|<1/2$ and $|x-y|<2$, then we have $|y-h|<5/2$. This gives us    
	$$\bigg|e^{\iota P(x,y)}-e^{\iota[R(x,y,h)+\sum\limits_{|\alpha|=m,|\beta|=l}a_{\alpha,\beta}(y-h)^{\alpha+\beta}]}\bigg|\leq C\sum_{|\alpha|=m,|\beta|=l}|a_{\alpha,\beta}||x-y|\leq C|x-y|.$$
	Thus, the term $T^{P}_{02,*}f$ can be dominated by the Hardy-Littlewood maximal function, i.e. we have  $T^{P}_{02,*}f(x)\lesssim Mf_{h}(x)$, where $f_{h}=f\chi_{B(h,\frac{9}{4})}$. Hence
	$$|\langle T^{P}_{02,*}f_h,g\chi_{B(h,\frac{1}{2})}\rangle |\lesssim \Lambda_{\mathcal{S}_{0},1,1}(f_{h},g).$$ 
	Since this estimate holds uniformly in $h$, using the same argument as in the previous section, we get that 
	\begin{eqnarray}
		|\langle T^{P}_{02,*}f,g\rangle|\lesssim \sup_{\mathcal S}\Lambda_{\mathcal{S},1,1}(f,g).
	\end{eqnarray}
	Next, we need to deal with the terms $T^P_{\infty}f$ and $T^P_{\infty,*}f$. 
	We decompose the operator $T^{P}_{\infty}$ as  $$T^{P}_{\infty}=\sum_{j\geq2}T^{P}_j.$$
	Since $\epsilon\geq 1$, we can choose $J\in \N$ such that $2^{J-1}\leq\epsilon <2^J$. Thus, we can write
	\begin{eqnarray*}
		T^P_{\infty,*}f(x)&\leq& \sup_{J\geq1}\int_{2^{J-1}\leq |y| <2^J} \frac{|\Omega(y)|}{|y|^n}|f(x-y)|~dy\\
		&&+\sup_{J \geq1} \sum_{j=\geq J+1}\left|\int_{2^{j-1} \leq|x-y|<2^{j}} e^{\iota P(x,y)} \frac{\Omega(x-y)}{|x-y|^n}f(y) ~dy\right|\\
		&\leq&\Vert \Omega\Vert_{L^{\infty}} Mf(x) + \sum_{j\geq2}|T^{P}_jf(x)|
	\end{eqnarray*}
	Also, note that $T^{P}_{\infty}f(x) \leq\sum\limits_{j\geq2}|T^{P}_jf(x)|$. Therefore, we need to prove sparse bounds for the operator $\sum\limits_{j\geq2}|T^{P}_jf(x)|$. 
	
	We work with the operator $T^{P}_j$ for each fixed $j\geq 2$. Let $Q$ be a dyadic cube with side length $l_{Q} = 2^{j+2}$ and recall the operator 
	\begin{equation*}
		I _{Q} f(x) := \int_{2^{j-1}\leq |y|<2^{j}} e ^{\iota P(x,x-y)} \frac{\Omega(y) }{|y|^{n}} \left(\chi_{\frac{1}{3}Q} f\right)(x-y)\; dy.  
	\end{equation*}
	Note that $ I_Q f $ is supported on $Q$. Since each dyadic grid $\mathcal{D}_{t}$ of cubes in $\mathbb{R}^{n}$ with side length $l_{Q}=2^{j+2}$ forms a partition of $\mathbb{R}^{n}$, the operator $T^{P}_{j}$ can be expressed as $T^{P}_{j}f:=\sum\limits_{Q\in \mathcal{D}_{t}}I_{Q}f$. Hence it is enough to work with the operators $I_{Q}$ instead of $T^{P}_{j}$. 
	
	We claim that for any $1<r<2$, we have 
	\begin{eqnarray}\label{lr}
		|\langle I_Q f, g \rangle| 
		&\lesssim& 2 ^{-\eta  j}|Q|\langle f\rangle_{Q,r}\langle g\rangle_{Q,r},
	\end{eqnarray}
	for some $\eta>0$.
	
	By summing over cubes $Q$ and  $j$  in the estimate above, we get the desired sparse bounds for $\sum\limits_{j\geq2}|T^{P}_jf(x)|$. 
	
	The estimate above is deduced by using interpolation between $L^2$ and $L^{\infty}$ estimates. The $L^2$ estimate follows from~Lemma~\ref{mainlemma} in the following manner. Applying  Minkowski's integral inequality in the equation \eqref{NQ} and using Lemma~\ref{mainlemma}, we get that  
	\begin{eqnarray*}
		\Vert I_{Q}f\Vert_{{2}}&\leq& \int_{\mathbb{S}^{n-1}}|\Omega(y')|\bigg(\int_{Y}\int_{\mathbb{R}}|N_{Q}f[z+\cdot y'](s)|^{2}~ds dz\bigg)^{1/2}d\sigma(y')\\
		&\lesssim& 2^{-\frac{j\delta}{2}}\Vert f\chi_{Q}\Vert_{{2}}\int_{\mathbb{S}^{n-1}}|\Omega(y')| \left|\sum_{\substack{|\alpha|=m \\|\beta|=l}} a_{\alpha,\beta}y'^{\alpha+\beta}\right|^{-\frac{\delta}{2m}}d\sigma(y')\\
		&\lesssim& 2^{-\frac{j\delta}{2}}\Vert \Omega\Vert_{{q}}\Vert f\chi_{Q}\Vert_{{2}}\left(\int_{\mathbb{S}^{n-1}} \left|\sum_{\substack{|\alpha|=m \\|\beta|=l}} a_{\alpha,\beta}y'^{\alpha+\beta}\right|^{-\frac{\delta q'}{2m}}d\sigma(y')\right)^{\frac{1}{q'}}.
	\end{eqnarray*}
	We choose $\delta \in(0,1]$ such that $\delta<\min \left\{\frac{m}{l}, \frac{2m}{(k+l)q'}\right\}$. Then using the estimate from Lemma~\eqref{poly}, we get that 
	\begin{eqnarray}\label{L2}
		\Vert I_{Q}f\Vert_{{2}}\lesssim2^{-\frac{j\delta}{2}}\Vert \Omega\Vert_{{q}}\Vert f\chi_{Q}\Vert_{{2}}.
	\end{eqnarray}
	Next, consider 
	\begin{eqnarray*}
		|I_{Q}f(x)|&=&\bigg|\int_{2^{j-1}\leq |y|<2^{j}}e^{\iota P(x,x-y)}\frac{\Omega(y/|y|)}{|y|^{n}}f\chi_{\frac{1}{3}Q}(x-y)~dy\bigg|\\
		&\lesssim&|Q|^{-1}\int_{2^{j-1}\leq |y|<2^{j}}|\Omega(y/|y|)||f\chi_{\frac{1}{3}Q}(x-y)|~dy\\
		&\lesssim& |Q|^{-1}\Vert \Omega\Vert_{\infty}\Vert f\chi_{Q}\Vert_{1}.
	\end{eqnarray*} 
	Interpolating between the $L^{\infty}$ estimate above and the $L^2$ estimate~\eqref{L2}  we get that for any $1<r<2$,
	\begin{eqnarray}
		\|I_Qf\|_{r'}\lesssim 2^{-\delta j(1-\theta)} |Q|^{-\theta}\|\Omega\|_{\infty} \|f\chi_{Q}\|_{r}
	\end{eqnarray}
	holds, where $1/r+1/r'=1$ and $\theta=1-\frac{2}{r'}$. 
	
	Therefore, for any $1<r<2$, we have 
	\begin{eqnarray*}
		|\langle I_Q f, g \rangle| &\lesssim& \|I _{Q} f\|_{r'} \|g\chi_{Q}\|_{r}\\ 
		&\lesssim& 2 ^{-\eta  j}  | Q|^{(-1+ \frac{2}{r'})} \|f\chi_{Q}\|_r \|g \chi_{Q}\|_{r}  \\
		&\lesssim& 2 ^{-\eta  j}|Q|\langle f\rangle_{Q,r}\langle g\rangle_{Q,r},
	\end{eqnarray*}
	where $\eta>0$.
	
	This completes the proof of Theorem~\ref{mainresult}. \qed
	
	\section{Proof of Theorem \ref{Lq}} 	
	The proof of Theorem~\ref{Lq} follows the same scheme as in the case of Theorem~\ref{mainresult}. We give a brief sketch here. As before, we decompose $$T^{P}_{\Omega}=T^{P}_0+T^{P}_{\infty},$$
	where $T^{P}_{\infty}:=\sum\limits_{j\geq2}T^{P}_j$.
	
	Sparse bounds for the local part $T^{P}_0$ follows using the same arguments as in the case of  Theorem~\ref{mainresult}.
	
	Next, in order to prove the desired estimates for the operator $T_j^P(f)$, we need to decompose $\Omega$ in the following manner. Consider $E_{0}=\{x^{\prime}\in \mathbb{S}^{n-1}:|\Omega(x^{\prime})|<1\}$ and
	$E_{k}=\{x^{\prime}\in \mathbb{S}^{n-1}:2^{k-1}\leq|\Omega(x^{\prime})|<2^{k}\}, ~~~~k \in \mathbb{N}$. 
	This allows us to write $\Omega(x')=\sum\limits_{k\geq 0}\Omega_{k}(x^{\prime}),$ where $\Omega_{k}(x^{\prime})=\Omega\left(x^{\prime}\right) \chi_{E_{k}}(x^{\prime})$. This gives us 
	$$T^{P}_jf(x)=\sum_{k=0}^{\infty} T_{j,k}f(x),$$
	where $T_{j,k}f(x)=\int_{2^{j-1}\leq |x-y|<2^{j}} e^{\iota P(x,y)} \frac{\Omega_{k}(x-y)}{|x-y|^{n}} f(y) dy$
	
	Therefore, we need to consider a further decomposition of the operator $I_Q$ in terms of  
	\begin{equation*}
		I _{Q,k} f(x) := \int_{2^{j-1}\leq |y|<2^{j}} e ^{\iota P  (x,x-y)} \frac{\Omega_{k}(y) }{|y|^{n}} (\chi_{\tfrac 13Q} f ) (x-y)\; dy, \qquad l(Q) = 2^{j+2}.  
	\end{equation*}
	We observe that the analogue of estimate~\eqref{lr} for the operator $I _{Q,k} f$ can be proved by imitating the arguments from that of $I _{Q} f$. Since this part can be completed without difficulty, we skip the details. For $1<r<2$ we get that 
	
	$$\Vert I_{Q,k}f\Vert_{r'}\lesssim 2^{-\frac{j\delta(1-\theta)}{2}} |Q|^{-\theta}\Vert \Omega_k\Vert_{q}^{1-\theta}\Vert \Omega_k\Vert_{\infty}^{\theta}\Vert f\chi_{Q}\Vert_{r},$$
	where $1/r+1/r'=1$ and $\theta=1-\frac{2}{r'}$.
	
	Since $\Vert \Omega_k\Vert_{\infty}\leq 2^k$ and $\Vert \Omega_k\Vert_{q}\leq \Vert \Omega_k\Vert_{\infty}|E_k|^{\frac{1}{q}}$, we have 
	$$\Vert I_{Q,k}f\Vert_{r'}\lesssim 2^{-\frac{j\delta(1-\theta)}{2}} |Q|^{-\theta}2^k|E_k|^{\frac{1-\theta}{q}}\Vert f\chi_{Q}\Vert_{r}.$$
	
	By summing over $k$, we obtain
	$$\Vert I_{Q}f\Vert_{r'}\lesssim 2^{-\frac{j\delta(1-\theta)}{2}} |Q|^{-\theta}\Vert \Omega\Vert_{L^{\frac{q}{1-\theta},1}\log L}\Vert f\chi_{Q}\Vert_{r}.$$
	
	We know that the embedding $L^{\frac{q}{1-\theta},1}\log L(\mathbb{S}^{n-1})\subseteq L^p(\mathbb{S}^{n-1})$ holds whenever $p\leq\frac{q}{1-\theta}$. Since $1-\theta=\frac{2}{r'}$ and $q>1$ we get that $r>(2p)'$. Putting all these estimates together for $1<r<2$, we get that   
	\begin{eqnarray*}
		|\langle I_Q f, g \rangle| &\lesssim& \|I _{Q} f\|_{r'} \|g\chi_{Q}\|_{r}\\ 
		&\lesssim& 2 ^{-\eta  j}  | Q|^{(-1+ \frac{2}{r'})} \|f\chi_{Q}\|_r \|g \chi_{Q}\|_{r}  \\
		&\lesssim& 2 ^{-\eta  j}|Q|\langle f\rangle_{Q,r}\langle g\rangle_{Q,r},
	\end{eqnarray*}
	where $\eta>0$. 
	Thus, by summing over $j$ and cubes $Q$, we get the desired result. This completes the proof of Theorem~\ref{Lq}. \qed
	
	\vspace{.25in}
	\noindent
	{\bf Acknowledgements:} The first author is supported by CSIR (NET), file no. 09/1020(0182)/2019-EMR-I. The second author acknowledges the support from Science and Engineering Research Board, Department of Science and  Technology, Govt. of India under the scheme Core Research Grant with file no. CRG/2021/000230. The third author is supported by the Republic of Korea NRF grant no.  2022R1A4A1018904 and BK21 Postdoctoral fellowship of Seoul National University.


\vspace{.3in}
	
	
	\begin{thebibliography}{99}
		\bibitem {Calderon}	A. P. Calder\'{o}n,  and A. Zygmund, {\it  On singular integrals}, Amer. J. Math. 78, (1956), 289--309.    
		\bibitem{Qassem2} H. Al-Qassem, L. Cheng, and Y. Pan, {\it A van der Corput type lemma for oscillatory integrals with Hölder Amplitudes and its applications,} J. Korean Math. Soc., 58(2), 487--499~(2021).		
		\bibitem{Qassem1} H. Al-Qassem, L. Cheng, and Y. Pan, {\it Oscillatory Singular Integral Operators with Hölder Class Kernels}. J. Fourier Anal. Appl. 25, 2141--2149 (2019).
		\bibitem{CC} S. Chanillo and M. Christ, {\it Weak $(1,1)$ bounds for oscillatory singular integral}, Duke Math.J. 55~(1987), 141--155.
		\bibitem{CKS1} S. Chanillo, D. Kurtz and G. Sampson, {\it Weighted $L^{p}-$estimates for oscillatory kernels,} Arkiv for Matematik 21(1983), 233-257.
		\bibitem{CKS2}  S. Chanillo, D. Kurtz and G. Sampson, {\it Weighted weak $(1,1)$ and weighted $L^{p}-$estimates for oscillatory kernels, } Trans. Amer. Math. Soc. 295(1986), 127-145.
		\bibitem{CT} Y. Chen and W. Tao, {\it Quantitative weighted estimates for oscillatory singular integrals with rough kernels}, Bull. Korean Math. Soc. 59 (2022), no. 1, 191--202.
		\bibitem{Conde} J. M. Conde-Alonso, A. Culiuc, F. Di Plinio  and Y. Ou, {\it A sparse domination principle for rough singular integrals}, Anal. PDE 10 (2017), no. 5, p. 1255--1284.
		\bibitem{Christ} M. Christ and Rubio de Francia, J.L. {\it Weak type (1, 1) bounds for rough operators}, II. Invent Math 93, (1988), 225--237.		
		\bibitem{Hytonen}	F. Di Plinio, T. P. Hytönen and K.  Li,  {\it Sparse bounds for maximal rough singular integrals via the Fourier transform},	Annales de l'Institut Fourier, Volume 70 (2020) no. 5, pp. 1871--1902.		
		\bibitem{Fujii} N. Fujii, {\it Weighted bounded mean oscillation and singular integrals, } Math. Japon. 22(1977/78), no. 5, 529--534.			
		\bibitem{Hofmann} S. Hofmann,  {\it Weak type (1, 1) boundedness of singular integrals with nonsmooth kernels}. Proc. Amer. Math. Soc. 103 (1988), 260--264.		
		\bibitem{JL} Y.S. Jiang and S. Lu, {\it Oscillatory singular integrals with rough kernel}, Harmonic analysis in China, 135--145, Math. Appl., 327~(1995).		
		\bibitem{KL} B. Krause and M. Lacey, {\it Sparse bounds for maximally truncated oscillatory singular integrals,} Ann. Sc. Norm. Super. Pisa Cl. Sci. (5), Vol XX (2020), 415--435.  
		\bibitem{Mena} M. Lacey and D. Mena Aris, {\it The sharp T1 theorem,} Houston J. Math. 43(1) (2017)111---127.	
		\bibitem{Lacey} M. T. Lacey and S. Spencer, {\it Sparse bounds for oscillatory and random singular integrals}, Newyork J. Math. 23~(2017), 119--131.		
		\bibitem{Lerner} A. K. Lerner, {\it A weak type estimate for rough singular integrals}, Rev. Mat. Iberoam. 35 (2019), no. 5, 1583--1602.		
		\bibitem{Lerner1} A. K. Lerner and F. Nazarov, {\it Intuitive dyadic calculus: the basics}, Expo. Math. 37  (2019), no. 3, 225--265.		
		\bibitem{Nieraeth} E. Lorist and Z. Nieraeth, {\it Sparse domination implies vector-valued sparse domination,}	Math. Z. 301 (2022), no. 1, 1107--1141.
		\bibitem{PhongStein} D. H. Phong and E. M. Stein, {\it Hilbert integrals, singular integrals, and Radon transforms I}, Acta Math. 157 (1986), no. 1-2, 99–157.
		\bibitem{Ricci} F. Ricci and E. M. Stein,  {\it Harmonic analysis on nilpotent groups and singular integrals, I. Oscillatory integrals,} J. Funct. Analysis, 73(1987), 179--194.	
		\bibitem{Sato} S. Sato, {\it Weighted weak type $(1,1)$ estimates for oscillatory singular integrals}, Studia Mathematica 141 (2000), 1--24.		
		\bibitem{Seeger}  A. Seeger,  {\it Singular integral operators with rough convolution kernels}, J. Amer. Math. Soc. 9 (1996), 95--105.			
		\bibitem{LZ} L. Shan Zhen and Z. Yan, {\it Criterion on $L^{p}$-boundedness for a class of oscillatory singular integral with rough kernels}, Rev. Mat. Iberoamericana 8 (1992), no. 2, 201--219.
		\bibitem{Stein} E. M. Stein, {\it Oscillatory integrals in Fourier analysis, Beijing lectures in Harmonic analysis } (Beijing, 1984 ) 307-355, The Annals of Mathematics Studies, 112, Princeton Univ. Press, 1986.	
		\bibitem{Wainger} E. M. Stein and S. Wainger, {\it Oscillatory integrals related to Carleson's theorem,} Math. Res. Lett. 8(2001), no. 5-6, 789--800.		
		\bibitem{Wilson} J. M. Wilson, {\it Weighted inequalities for the dyadic square function without dyadic $A_{\infty}$}, Duke Math. J. 55(1987), no. 1, 19--50.
	\end{thebibliography}
\end{document}